\documentclass[11pt]{amsart}
\usepackage{amsmath}
\usepackage{amsfonts}
\usepackage{amssymb}
\newtheorem{thm}{Theorem}[section]

\newtheorem{cor}[thm]{Corollary}

\newtheorem{remark}[thm]{Remark}
\newtheorem{lemma}[thm]{Lemma}

\newtheorem{defn}[thm]{Definition}

\newcommand{\bb}[1]{\mathbb{#1}}
\newcommand{\cl}[1]{\mathcal{#1}}

\begin{document}

\title[Invariant spaces under $S^n$]{A Finite Multiplicity Helson-Lowdenslager-De Branges Theorem}

\author[S.~Lata]{Sneh Lata}
\address{Department of Mathematics, University of Houston,
Houston, Texas 77204-3476, U.S.A.}
\email{snehlata@math.uh.edu}

\author[M.~Mittal]{Meghna Mittal}
\address{Department of Mathematics, University of Houston,
Houston, Texas 77204-3476, U.S.A.}
\email{mittal@math.uh.edu}

\author[D.~Singh]{Dinesh Singh}
\address{Department of Mathematics, University of Delhi,
Delhi 110007, INDIA}
\email{dineshsingh1@gmail.com}

\date{September 23, 2009}

\begin{abstract}
This paper proves two theorems. The first of these
simplifies and lends clarity to the previous characterizations of the
invariant subspaces of $S$, the operator of multiplication by the coordinate
function $z$, on $L^2(\mathbb{T};\mathbb{C}^n)$, where $\mathbb{T}$ is the unit circle,
by characterizing the invariant
subspaces of $S^n$ on scalar valued $L^p$ ($0<p\le\infty$) thereby
eliminating range functions and partial isometries. It also gives
precise conditions as to when the operator shall be a pure shift and
describes the precise nature of the wandering vectors and the doubly
invariant subspaces. The second theorem describes the contractively
contained Hilbert spaces in $L^p$ that are simply invariant under $S^n$
thereby generalizing the first theorem.\end{abstract}

\maketitle


\section{Introduction}
The Helson-Lowdenslager invariant subspace theorem, \cite{HL}, describes the simply 
invariant subspaces of the operator $S$ of multiplication by the coordinate function 
$z$ on $L^2$ thereby generalizing in a highly non-trivial fashion Beurling's famous 
invariant subspace theorem \cite{Beu}. The theorem and the method invented by Helson 
and Lowdenslager are important for a variety of reasons and we refer to \cite{DPRS}, \cite{Gam},  
\cite{Hel1}, \cite{Hof2}, \cite{PS}, \cite{PS1} for such details. The doubly 
invariant subspaces of $S$ are described by Wiener's theorem. We refer to \cite{Hel2} 
for details. Between them, these two theorems describe the class of all invariant subspaces 
for $S$ on scalar valued $L^2$ of the unit circle. The description of such invariant 
subspaces of $S$ has been generalized and extended in many directions and we refer 
to \cite{Gam}, \cite{Hel2}, \cite{Hel3}, \cite{PS}, \cite{Red}, \cite{RaRo}. We also refer 
to \cite{Sin} for several other useful references. Among these fruitful and interesting 
generalizations is the extension of these theorems to the context of functions on 
the unit circle taking values in an arbitrary Hilbert space H. Such a generalization 
can be found among other places in Helson's book \cite{Hel2} and in \cite{RaRo} as also 
in \cite{Ohn}, \cite{San}, \cite{Sri} . The tenor of these generalizations involves the 
description of the decomposition of these invariant subspaces into the direct sum of 
two invariant subspaces one of which is the shift part and the other being doubly 
invariant. Unlike the scalar case, such generalizations do point out the fact that 
the two subspaces-the shift part and the doubly invariant can co-exist simultaneously 
to together add up to the entire invariant subspace which shall remain simply invariant. 

The purpose of this paper is to prove two theorems. The first of these looks anew at 
the solved problem of characterizing the invariant subspaces of $ S$ on the Lebesgue 
space $L^2$ whose functions take values in a finite dimensional Hilbert space $H.$ 
More precisely, unlike the previous characterizations, we characterize-without taking recourse to vectorial concepts such 
as that of a range function and without stepping outside the scalar situation-the 
most general form of an invariant subspace of the operator of multiplication by $z^n$ 
in the classical Lebesgue space $L^p, \ 0< p\le \infty,$ on the unit circle. This includes 
describing the precise nature of the shift part of a simply invariant subspace by 
characterizing in an explicit and simple form the nature of its wandering vectors and 
by giving exact conditions under which a simply invariant subspace cannot contain a 
reducing subspace. We also show when the reducing part of a simply invariant subspace 
will not be zero and in that event we describe the precise nature of the reducing part 
in terms of the wandering vectors of the shift part.  

The second theorem, in this paper, characterizes the class of all Hilbert spaces that 
are contractively contained in $L^q, \ 0< q \le 2,$ and on which the operator of multiplication 
by $z^n$ acts isometrically. By imposing a natural condition, we give a precise description 
of such spaces which also generalizes our first theorem as well as the main theorems in 
\cite{PS} and \cite{Red}, \cite{SS}, \cite[Page 26]{Hel2} and a theorem of de Branges \cite{dBra}.\\

 This paper is organized as follows. In section 2, we give some of the basic notations, 
 definitions and the statement of the main results. In Section 3, we prove some of the preliminary 
 results followed by the proof of the first main theorem. Section 4 deals with some basic results 
 concerning De Branges spaces and it then presents the proof of the second main theorem. Also, we prove a direct 
 generalization of the characterization, as given in \cite{PS}, \cite{Red}, of those Hilbert spaces which are boundedly contained in $L^p$ 
 and are doubly invariant under $S^n$ Of course while we deal with the case of a general positive integer these papers deal with the case $n=1.$  
 We would like to mention that the results proved in a De Branges setting give an alternate 
 characterization theorem for the classical case. It would be interesting to compare them.
 Finally towards the end of this section we give some final remarks and results which follow as
 consequences of the results proved earlier in the same section.

 \section{Notations, Terminology and Statement of the Main Results}
 We shall denote the unit circle in the complex plane by $\bb T$
 and the open unit disk by $\bb D$. For any $0 < p \le \infty, \ L^p$ and $H^p$ 
 shall denote the familiar Lebesgue and Hardy spaces respectively on $\bb T$.
 Recall that $ L^p$ and consequently also $H^p,$ are Banach spaces 
 with the usual norm which is given by $\|f\|_p=(\frac{1}{2\pi}\int_{0}^{2\pi}|f|^pdm)^{1/p}$ 
 when $0< p < \infty$ and $\|f\|_{\infty}=esssup|f|$ when $p=\infty.$
 In particular, when $p=2$ these are Hilbert spaces with the inner product given by 
 $<f,g>= \frac{1}{2\pi}\int_0^{2\pi}f\overline{g}dm.$
 For a fixed positive integer n, $L^p(z^n)$ shall denote the closed 
 linear span of $\{z^{kn}: k \in \bb Z\}$ in $L^p$ and $H^p(z^n)$ will stand for $L^p(z^n) \cap H^p.$
 
 A Hilbert space $\cl M$ is said to be {\em  boundedly contained} in a Hilbert space
 $\cl H$ if $\cl M$ is a vector subspace of $\cl H$ (in the algebraic sense) 
 and if the inclusion map, \ i.e. $\|x\|_{\cl H}\le C \|x\|_{\cl M}$ for all z in $\cl M$
 and some constant $C.$ When $C=1,$ we say that $\cl H$ is {\em contractively contained} in $\cl M.$
 For further details on all of the above we refer to \cite{Hel1},\ \cite{HL} and \cite{Hof1}.\\

 For a fixed positive integer n, $S^n$ shall denote the operator of multiplication 
 by $z^n$ where $z$ is the coordinate function on $\bb T.$
 It is easy to see that $S^n$ is an isometry on $L^2$ and hence also on $H^2.$
 A closed subspace $\cl M$ of $L^p, \ 0 < p \le \infty$ ( weak star closed when $p= \infty$) 
 is said to be {\em simply invariant} under $S^n$
 if $S^n(\cl M) \subsetneq \cl M$ and {\em doubly invariant} if $S^n(\cl M)=\cl M.$
 For our purposes, a Hilbert space $\cl M$ which is boundedly contained in $L^p$
 will be said to be {\em simply} ({\em doubly}) {\em invariant} if $S^n(\cl M)\subsetneq \cl M \ (S^n(\cl M)=\cl M).$\\

 If T is an operator on some Hilbert space $\cl K,$ then we let $\cl R(T)$ denote 
 the range of $T$, which is a (not necessarily closed) subspace of $\cl K.$
 However, if we endow $\cl R(T)$ with the norm, $\|h\|_{\cl R(T)}=\inf\{\|k\|_{\cl K}: Tk=h\},$ 
 then $\cl R(T)$ is a Hilbert space in this norm called the {\em range space} of $T$
 and it is boundedly contained in $\cl K.$ When $T$ is a contraction then the range 
 space of $T$ is contractively contained in $\cl K.$ \\

 Before proceeding with some more terminology we note that we can write 
 \[ L^p = L^p(z^n)\oplus z L^p(z^n) \oplus \cdots \oplus z^{n-1}L^p(z^n).\] 
 Obviously, by the above direct sum we mean the Hilbert space direct sum when $p=2$ and algebraic direct sum otherwise.
 The following definition concerns $L^{\infty}$ functions. It is a direct analogue of 
 the definition of n-inner functions in $H^{\infty}$ introduced by the third author and 
 Thukral in \cite{ST}. 

 \begin{defn} A function $\phi \in L^{\infty}$ is called n-unimodular if $ \sum_{i=1}^{n}|\phi_{i}|^2=1$ a.e.
  where $\phi= \sum_{i=1}^{n}z^{i-1}\phi_{i}$ with $\phi_{i} \in L^{\infty}(z^n) \ \forall \ i.$
 \end{defn}

 Lastly, we define an operator on $L^2$ which is an 
 analogue of the multiplication operator usually denoted by $M_{\phi}$ where $\phi \in L^{\infty}.$
 Take a matrix $\psi=(\psi_{ij}) \in M_{mn}(L^{\infty})$ and 
 define an operator
 on $L^2$ via $A_{\psi}:L^2 \to L^2$ by \[A_{\psi}(f)=\sum_{i=1}^m z^{i-1}\left(\sum_{j=1}^n\psi_{ij}(z)f_j\right)\] 
 where $f=\sum_{i=1}^nz^{i-1}f_i$ with $f_i\in L^2(z^n).$  

\noindent It is easy to see that $A_{\psi}$ is a bounded operator on $L^2$ and $n=m=1$ 
 gives us the usual multiplication operator on $L^2.$ Note that the kernel of the 
 operator $A_{\psi}$ is a doubly invariant subspace of $L^2$ under $S^n.$ 
 We adopt the notation, $\cl K^{\psi}_{\cl M}$ for the kernel of the operator 
 $A_{\psi}$ restricted to the subspace $\cl M$ of $L^2.$ \\
  By $\overline{\psi}$ we mean the matrix $(\overline{\psi_{ij}}),$ where $\overline{\psi_{ij}}$ is 
 the complex conjugate of the function $\psi_{ij}.$

 For our purposes, $\cl K^{\psi}_{\cl M}$ is specific in nature. 
 To see this, let  $\{\psi_1, \cdots, \psi_r\}$ be a set of n-unimodular functions in $L^{\infty}$ 
 and let $\psi=(\psi_{ij}),$ where $\psi_i= \sum_{j=1}^nz^{j-1}\psi_{ij}, \ \psi_{ij}\in L^{\infty}(z^n).$
 Then it can be easily seen that 
 \[\cl K^{\psi}_{\cl M}= \{f \in \cl M: \sum_{j=1}^n\psi_{ij}f_j=0 \  \ \forall \ 1 \le i\le r\},\]
where $f=\sum_{j=1}^nz^{j-1}f_j, \ f_j\in L^2(z^n).$

 If we let $\cl M=L^2$ and replace $\psi$ by $\overline{\psi}$ then 
 as noted before, $\cl K^{\overline{\psi}}_{L^2}$ is the doubly invariant subspace of $L^2$ under $S^n.$
 In fact, it is the maximal invariant subspace of $L^2$ that is 
 orthogonal to $\psi_i H^2(z^n) \ \forall \ 1\le i \le n.$\\
 
Keeping this motivation in mind, we define \[\cl K^{\psi}_{\cl M(p)}= \{f \in \cl M: \sum_{j=1}^n\psi_{ij}f_j=0 \  \ \forall \ 1 \le i\le r\},\] where $f=\sum_{j=1}^nz^{j-1}f_j, \ f_j\in L^p(z^n)$ and  $\cl M$ is any subspace of $ L^p.$\\

 We are now in a position to give the statement of our main results.\\

\noindent {\bf Theorem A.}\label{one} {\it  Let $\cl M$ be a simply
 invariant subspace of $L^p, \ 0< p \le \infty$ under $S^n$. 
 Then the most general form that $\cl M$ can assume is 
\[\cl M=\sum_{i=1}^r \oplus \phi_iH^p(z^n) \oplus \cl K_{\cl M}^{\overline{\phi}}\]
 where 
\begin{itemize}
 \item[(i)] $r \le n, $
 \item[(ii)] each $\phi_i$ is a n-unimodular function not vanishing on any set of positive measure,
 \item[(iii)]$\|\phi_if\|_p=\|f\|_p $ for every $ f \in H^p(z^n),$ and 
 \item[(iv)] $ \phi=(\phi_{ij}) \in M_n(L^{\infty}(z^n)), \ \phi_i= \sum_{j=1}^nz^{j-1}\phi_{ij}.$
\end{itemize}

  \indent Moreover, if $r=n$ then $\cl K_{\cl M}^{\overline{\phi}}=\{0\}$ and if $r < n$ then 
 there exists infinitely many nonzero doubly invariant subspaces of 
 $\cl K^{\overline{\phi}}_{L^2}$ which when appended to $\sum_{i=1}^r \oplus \phi_iH^2(z^n)$
 form a simply invariant subspace of $L^p.$ }
\\

 The second of our two main theorems shows that under similar conditions as considered in \cite{PS} 
 and \cite{Red}, there are non-trivial Hilbert spaces contractively contained in $L^p,$ for $ 0< p \le 2.$ 
 The analogous result for the case $2<p \le \infty$ as obtained by the third author and 
 Paulsen in \cite{PS} can be found in a later section which asserts that there is no 
 non-trivial simply invariant Hilbert space under $S^n$ which is  contractively contained in $L^p.$ We shall adopt the notation $\frac{2q}{2-q} =\infty$ when $q=2.$\\

 \noindent{\bf Theorem B.}  {\it Let $\cl M \neq 0$ be a simply invariant Hilbert
 space contractively contained in $L^q, \ 0 < q \le 2$ and on which $S^n$ 
 acts isometrically. Further, suppose that  $\exists \ p, \ 2\le p\le \frac{2q}{2-q}  $ and a  $\delta >0 $
 such that \[ \|f\|_\cl M \leq\ \delta\|f\|_p    \ \ \ \ \ \ \ \  \forall \ f \in \cl M \cap L^p .\]
 Then there exists an orthonormal set, $\{\psi_i\}_{i=1}^r \subseteq \cl M \cap L^\infty, \ r \le n,$
 each of which is nonzero a.e. such that 
   \[\cl M = \sum_{i=1}^r\oplus\psi_i H^2(z^n)\oplus  \cl R(M_{h}A_{\phi}),\]
where
\begin{enumerate}
 \item[(i)] $\phi \in M_n(L^{\infty}(z^n))$ and $g \in L^{2q/(2-q)}$ is a strictly positive function such that
\begin{equation*} 
h=\left\{\begin{array}{rl} g  & {\rm when } \ 0 <q<2 \\
1 & {\rm  when } \ q=2
\end{array} \right.
\end{equation*}
 \item[(ii)]  $\|\psi_i f\|_{\cl M}= \|f\|_2 $ for every $ f \in H^2(z^n)$ and $ 1 \le i \le r,$
 \item[(iii)] $(\sum_{j=1}^r|\psi_{ij}|)^{-1} \in L^s$ and $ \  \|(\sum_{j=1}^r|\psi_{ij}|)^{-1}\|_s \le \delta   \ \ {\rm  where } \
  \psi_i= \sum_{j=1}^r z^{j-1}\psi_{ij}  \\  {\rm and \ }
 s=\left\{\begin{array}{rl} \frac{2p}{p-2} & {\rm if} \ 2 < p \le \frac{2q}{2-q} \\
 \infty & {\rm  if } \ p=2
 \end{array} \right.$  
 \end{enumerate}  }
\indent Moreover, if $r=n$ then $\cl R(M_{h}A_{\phi})=\{0\}.$\\

\section{Proof of Theorem A.}

We begin this section with some of the preliminary results required to give the proof of the first main theorem. 
\begin{lemma}\label{ortho}  Let $\cl M$ be a subspace of $L^2$ which 
is simply invariant under $S^n.$ If we denote $\cl N=\cl M \ominus S^n\cl M,$
then
\begin{enumerate}
\item[(i)]\label{orthonormal} $\{\phi z^{kn}\}_{k\in\bb Z}$ is an orthonormal set
for each unit vector $\phi \in \cl N,$
\item[(ii)] if a unit vector $\phi \in \cl N$, then $\phi$ is n-unimodular,
\item[(iii)] $ \cl N \subseteq  L^\infty,$
\item[(iv)] if $\{\phi,\psi\}$ is an orthogonal set in $\cl N$,
then $\sum_{j=1}^{n}\phi_j\bar{\psi_j}= 0$ a.e., 
where $ \phi= \sum_{j=1}^{n}z^{j-1}\phi_j$ and $\psi= \sum_{j=1}^{n}z^{j-1}\psi_j,$
\item[(v)] elements of $\cl N$ cannot vanish on sets of positive measure.
\end{enumerate}
\end{lemma}

\begin{proof} Given that $\cl M= \sum_{k=0}^{\infty}\oplus S^{kn}(\cl N) \bigcap_{k\ge 0}S^{kn}\cl M$
(by the decomposition of isometries as in \cite[page 109]{Hof1}) the conclusion of (i) is immediate.
In order to prove (ii), we let $\phi\in \cl N, \ \|\phi\|_2 = 1.$ 
Then by (i), $\phi \perp \phi z^{kn} \ \forall \ k \neq 0.$ This yields 
that $\tfrac{1}{2\pi}{\int_0^{2\pi} |\phi|^2\*z^{kn}} = 0 \  \forall \ k \neq 0.$ 
If we write $\phi= \sum_{j=0}^{n-1}z^j\phi_j $ with $\phi_j \in H^2(z^n)$ 
then $\sum_{j=1}^{n}|\phi_j|^2=\delta$, for some constant $\delta.$ 
It is easy to see that $\delta=1$ since $\|\phi\|_{2}=1.$
This clearly shows that $\phi_i \in L^{\infty} \ \forall \  i$ and hence also proves (iii).\\
\indent To prove (iv), let $\{\phi, \psi\}$ be an orthonormal set in $\cl N.$ 
Then the decomposition of isometries gives us that $\left<\phi, z^{kn}\psi \right>=0 \ \forall \ k \in \bb Z.$
Let $\phi=\sum_{j=1}^{n} z^{j-1}\phi_j $ and $ \psi=\sum_{j=1}^{n} z^{j-1}\psi_j ,$ 
where as usual $\phi_j, \ \psi_j$ are in $H^2(z^n),$ then 
this yields that $\left<\sum_{i=1}^{n}\phi_i\bar{\psi_i}, z^{kn} \right> =0 \ \forall \ k.$ 
Also, note that $\left<\sum_{i=1}^{n}\phi_i\bar{\psi_i}, z^{2n-1} \right> =0 \ \forall \ n \in \bb Z,$ 
since $\phi_i$ and $\psi_i$ are even functions. \\
\indent Finally we prove (v) by contradiction. Take a non 
zero $\phi$ in $\cl N$ and suppose there exists a set of positive measure $A$ in 
$\bb T$ such that $\phi$ vanishes on $A$. \\
Define  
\begin{equation*}
k_m=\left\{\begin{array}{rl} m & {\rm on} \ B\\
1 & {\rm on} \ B^c
\end{array} \right.
\end{equation*}
where $B=\{z^n :z \in A \}.$ Note that $k_m \in L^{\infty}$ for each $m.$ \\
Let $t_m(z)=k_m(z^n)$ and  $h_m={\rm exp}(t_m+i\tilde{t}_m),$ 
where $\tilde{t}_m$ denotes the harmonic conjugate of $t_m,$ 
so that $h_m\in H^{\infty}(z^n)$ for all $m.$ 
Let $ h_m = \sum_{k=0}^{\infty}\alpha_kz^{kn}$ then by (i) we get that 
\[\| \phi \sum_{k=0}^l \alpha_kz^{kn}\|_2  = (\sum_{k=0}^l|\alpha_k\ |^2)^{1/2}  = \|\sum_{k=0}^l\alpha_k z^{kn}\|_2. \]
Note that $\phi \sum_{k=0}^l \alpha_kz^{kn}\in \cl M$ and is Cauchy in $L^2.$
This shows that $\phi h_m\in \cl M$ and $\| \phi h_m\|_2 = \|h_m\|_2.$ 
By the definition of $h_m,$ we get that $\|\phi h_m\|_2 < K$ for some constant $K$ 
independent of $m$ and $\|h_m\|_2 > \sqrt{{\rm exp}(2m) m(A)}$ which is indeed a contradiction.
This completes the proof of the lemma.
\end{proof}

\begin{lemma}\label{matrix}Let $\{\phi_1,\dots,\phi_n\}$ be a set of n-unimodular 
functions such that $\phi_i\perp\phi_jH^2(z^n)$ for each $i \ne j.$ If we write 
$\phi_i(z)=\sum_{j=1}^nz^{j-1}\phi_{ij}(z)$ with $\phi_{ij}\in L^2(z^n) \ \forall \ i,\ j \in \{1, \cdots, n\},$
and define 
\begin{displaymath}
A_r(z)=\left( \begin{array}{ccc} \phi_{11}(z)& \cdots & \phi_{1n}(z)\\
\vdots &\vdots&\vdots\\
\phi_{j1}(z)& \cdots& \phi_{jn}(z)\\
\vdots &\vdots&\vdots\\
\phi_{r1}(z)& \cdots& \phi_{rn}(z)
\end{array}\right)
\end{displaymath}
 then
\begin{itemize}
\item[(i)] $A_r(z)A_r(z)^*= I_r$ for almost all $z.$ In particular, 
if $r=n$ then $A_n(z)$ is unitary for almost all $z.$ 
\item[(ii)] there exists no nonzero $\psi \in L^2$ such that 
$\psi \perp\phi_i L^2(z^n)$ for each $i.$ 
\end{itemize}
\end{lemma}

\begin{proof}
It follows by the logic used to deduce Lemma~\ref{ortho}(iv) that 
$\sum_{l=1}^n\phi_{il}\overline{\phi_{jl}}=0 \ a.e. \ \forall \ i \ne j.$
This together with the hypothesis that each $\phi_i$ is n-unimodular forces $A_r(z)A_r(z)^*=I_{r} \ a.e.$ 
and hence completes the proof of (i). \\
\indent Assume that $\psi \perp \phi_iL^2(z^n)$ for each $i$ and let 
$\psi=\sum_{j=1}^nz^{j-1}\psi_j$ where $\psi_j \in L^2(z^n) .$
Then $\sum_{j=1}^n\phi_{ij}\overline{\psi_j}=0 \ a.e. \ \forall \ 1 \le i \le n.$ Thus, 
\begin{displaymath}
A_n(z){\left( \begin{array}{c} \overline{\psi_{1}(z)}\\
\vdots\\
\overline{\psi_{j}(z)}\\
\vdots \\
\overline{\psi_{n}(z)}
\end{array}\right)}=\left( \begin{array}{c} 0\\
\vdots\\
0\\
\vdots \\
0
\end{array}\right) \ a.e.
\end{displaymath}
It follows from this and (i) that $\psi=0 \ a.e.$ and this completes the proof of the result.
\end{proof}

\begin{lemma}\label{zero}
Let $\cl M$ be a simply invariant subspace of $L^p,\ 0 < p \le \infty$ under $S^n$ and 
let $f\in L^p.$ If $z^{nm}f\in \cl M$ for each $m \ge 1$ and there exists $k\in H^{\infty}(z^n)$ 
with $\langle{k,1}\rangle \ne 0$ such that $kf\in \cl M,$ then $f\in \cl M.$  
\end{lemma}
\begin{proof}Let $\langle{k,1}\rangle=a_0,$ then $k-a_0\in H^{\infty}(z^n).$ If we denote the $m^{th}$ Cesaro mean of $k-a_0$ by $k_m,$ then each $k_m$ is a polynomial in $z^n$ that converges to $k \ a.e.$ and $ ||k_m||_{\infty}\le ||k||_{\infty}.$ Clearly $k_mf \in \cl M \ \forall \ m \ge 1$ since $\langle{k_m-a_0,1}\rangle=0$ and $z^{nm}f \in \cl M \ \forall \ m \ge 1.$ Note that $k_mf \rightarrow (k-a_0)f$ in $L^p.$ From which it follows that $f\in \cl M.$ 
\end{proof}

\begin{lemma}\label{greater than 2}
Let $\cl M$ be simply invariant subspace of $L^p, \ 0 < p \le \infty$ under $S^n.$ Then $\overline{M}^{L^2}\cap L^p= M.$
\end{lemma}
\begin{proof}
Let $f \in \overline{M}^{L^2}\cap L^p.$ Then $\exists$ a sequence $\{f_k\} \subseteq M$ such that $f_k \rightarrow f$ in $L^2.$ 
Note that there exists $f_{kj}, \ f_j \in L^2(z^n)$ such that $f_k=\sum_{j=1}^nz^{j-1}f_{kj},\ f=\sum_{j=1}^nz^{j-1}f_j$ 
and $f_{kj} \rightarrow f_j$ in $L^2$ for each $ j.$ 
By Lemma 4.1 in \cite{Gam}, there exists an uniformly bounded 
$h_{kj} \in H^{\infty}$ which converges to 1 a.e. and 
$h_{kj}f_{kj} \rightarrow f_j$ in $L^p \ \forall \ j=1,\cdots,n.$ 
Observe that we can easily assume that $h_{kj} \in H^{\infty}(z^n) \ \forall \  j.$
Thus, it follows that $\prod_{j=1}^n h_{kj}f_k$ is a sequence in $\cl M$ which converges to $f$ in $L^p$ 
and hence the result.
\end{proof}

\begin{remark}  The above lemma yields that there does not exist any proper $S^n$ 
invariant subspace of $L^p$ which is dense in $L^2.$
\end{remark}

We are now in a position to prove the first of our two main theorems.\\

\noindent{\bf{\em Proof of Theorem A.}} We break the proof into three cases.\\
{\bf Case (1).} Assume $p=2.$ By using the 
decomposition of isometries [9, page 109] we get that
\[\cl M=\cl M_1 \oplus \cl K \] where $\cl K=\bigcap_{k \ge 0}{S^{nk}\cl M},
 \ \cl M_1=\sum_{k\ge 0}S^{kn}\cl N$ and $\cl N=\cl M \ominus S^n\cl M.$ 
Clearly, $\bigcap_{k\ge0}S^{kn}\cl M_1=\{0\}$ and $\bigcap_{k\ge 0}S^{kn} \cl K =\cl K.$\\
To establish the result, we claim that the dimension of $\cl N \le n$.\\ 
Suppose there exists a set of n+1 orthonormal vectors $\{\phi_1,\dots,\phi_{n+1}\}$ 
in $\cl N.$ Then by (ii) and (iv) of Lemma~\ref{orthonormal} we find that
each $\phi_i$ is n-unimodular and $\phi_{n+1}\perp\phi_iL^2(z^n) \ \forall \ 1\le i \le n.$ 
Thus, by Lemma ~\ref{matrix}(ii), it follows that $\phi_{n+1}=0$ which is a contradiction to 
the fact that $||\phi_{n+1}||=1.$  
This proves that \[\cl M_1=\sum_{i=1}^r\oplus{\phi_i H^2(z^n)}\] 
where $r\le n, \ {\rm and} \  \{\phi_1,\dots, \phi_r\}$ is an orthonormal basis of $\cl N.$
The rest of the properties of $ \{\phi_i\}_{i=1}^r$ follow from Lemma~\ref{ortho}.\\
\indent Lastly, to see $\cl K=\cl K^{\overline{\phi}}_{\cl M},$ it is enough to note that  $\cl K$ 
is doubly invariant and is orthogonal to $\phi_iH^2(z^n)$ for each $ i.$

\indent Finally, if $r=n$ then Lemma~\ref{matrix}(ii) asserts that $\cl K^{\overline{\phi}}_{\cl M}=\{0\}$, 
and if $r < n$ then by Lemma~\ref{matrix}(i) there exist infinitely many 
functions in $\cl K^{\overline{\phi}}_{L^2}.$ 
Consequently, there exist infinitely many nonzero doubly invariant subspaces, $\cl K^{\overline{\phi}}_{M'} \subseteq \cl K^{\overline{\phi}}_{L^2}$ such that 
$\sum_{i=1}^r\oplus{\phi_i H^2(z^n)}\oplus \cl K^{\overline{\phi}}_{M'}$ is a simply invariant subspace of $L^2.$ 
This completes the proof of the case $p=2.$\\

\noindent{\bf Case (2).} Assume $0 < p < 2.$ 
We first claim that \[\cl M \cap L^2 \neq 0.\]
Let $f \in \cl M,$ then $|f|^{p/2}$ is in $L^2$. It is easy to see that 
\[ g(z)=\tfrac{\sum_{i=0}^{n-1}|f|^{p/2}(\alpha^iz)}{n} \]
belongs to $L^2(z^n),$ where $\alpha\neq 1$ is an $n^{th}$ root of unity. 
If h(z) denotes the harmonic conjugate of g(z), then 
$h \in L^2(z^n)$ and hence $g+ih$ is in $ H^2(z^n).$ 
Set $k={\rm exp}[-(g+ih)].$ Then $k$ is an outer function in $ H^{\infty}(z^n)$. 
It is immediate to see that $kf \in  L^{\infty}$ and hence in $L^2$.
We show that $kf$ is in $\cl M$. Let  $k_m$ be the $m^{th}$ Cesaro mean of $k$.
Then each $k_m$ is a polynomial in $H^2(z^n)$ and so $k_{m}f$ is in $\cl M$ for every m. 
Further, $\|k_m\|_{\infty} \le \|k\|_{\infty}$ and $k_m\rightarrow k$ 
almost everywhere.  Thus $\|k_m f-kf\|_p \rightarrow 0$.
This shows that $k f$ is in $\cl M$ and thus
$\cl M \cap L^2$ is a 
non-trivial subspace of $ L^2$. \\
\indent Note that $\cl M \cap L^2 $
is closed in $ L^2$ because $\cl M$ is closed in $L^p$. 
Also $\cl M \cap L^2 $ is a simply invariant subspace of $ L^2$ 
under $S^n,$ so by Case (1) we conclude that 
\[\cl M \cap L^2= \sum_{i=1}^r\oplus\phi_i H^2(z^n)\oplus \cl K_{\cl M \cap L^2}^{\overline{\phi}} \]
where $\{\phi_i\}_{i=1}^r$ is a set of n-unimodular orthonormal vectors 
that do not vanish on any set of positive measure and $\phi$ is the corresponding matrix 
in $L^{\infty}(z^n)$ as obtained in the Case (1). \\

We claim that \[\cl M = \sum_{i=1}^r\oplus\phi_i H^p(z^n)\oplus \cl K_{\cl M}^{\overline{\phi}} .\]
Since $\cl M$ is invariant under $S^n$ and $\phi_i\in \cl M,$ therefore, 
\[\cl M \supseteq \sum_{i=1}^r\oplus\phi_i H^p(z^n)\oplus \cl K_{\cl M}^{\overline{\phi}} .\]
To establish the other containment, let $f \in \cl M.$ 
Then as shown above, there exists an outer function 
$k\in H^{\infty}(z^n)$ such that $kf\in M\cap L^{\infty}.$ 
So we can write  $kf=\sum_{i=1}^r\phi_i h_i+s$ for some $h_i$ in $H^2(z^n)$ 
and $s\in \cl K_{\cl M \cap L^2}^{\overline{\phi}}.$ 
Note that
\[\tfrac{\sum_{j=0}^{n-1}|\sum_{i=0}^{n-1}\alpha^{ij}kf(\alpha^iz)|^2}{n^2}=\sum_{i=1}^r|h_i(z)|^2+\sum_{i=1}^n|s_i|^2,\] where $s=\sum_{i=1}^n z^{i-1}s_i, \ s_i\in L^2(z^n).$\\
Since $k\in H^2(z^n)$, we have that
\[\sum_{i=1}^r|h_i(z)|^2+\sum_{i=1}^n|s_i|^2=\tfrac{|k(z)|^2}{n^2}(\sum_{j=0}^{n-1}|\sum_{i=0}^{n-1}\alpha^{ij}f(\alpha^iz)|^2)\]
$\Longrightarrow$
\[\sum_{i=1}^r|\tfrac{h_i(z)}{k}|^2+\sum_{i=1}^n|\tfrac{s_i}{k}|^2=\tfrac{\sum_{j=0}^{n-1}|f(\alpha^iz)|^2}{n} 
\le \tfrac{(\sum_{j=0}^{n-1}|f(\alpha^iz)|)^2}{n}\]
This together with the fact that $k$ is outer shows 
that $\tfrac{h_i}{k} \in H^p(z^n)\ \forall \ i,$ which further 
implies that $\sum_{i=1}^r \oplus \phi_i \tfrac{h_i}{k} \in 
\sum_{i=1}^r \oplus\phi_i H^p(z^n)\subseteq \cl M.$ \\

\indent Lastly, we need to show that $\frac{s}{k}\in 
\cl K_{\cl M}^{\overline{\phi}}.$ Note, by the proof of Case (1) we have that 
$\cl K_{\cl M\cap L^2}^{\overline{\phi}}=\bigcap_{l \ge 0}{S^{nl}(\cl M \cap L^2)}.$ 
Thus, there exists $s_l\in \cl M \cap L^2$ such that $s=z^{ln}s_l \ \forall \ l \ge 0.$ 
We shall prove that $\frac{s_l}{k} \in \cl M \ \forall  \ l \ge 0 $ by induction on $l$.
Clearly, from the above paragraph $\frac{s_0}{k} \in \cl M.$
Also it can be easily seen that $\frac{s_1}{k}$ satisfies all 
the hypotheses of Lemma~\ref{zero} and thus $\frac{s_1}{k}\in \cl M.$ 
Note that $s_{l-1}=z^ns_l,$ thus it is immediate that $\frac{s_l}{k} \in \cl M$ if $\frac{s_{l-1}}{k}\in \cl M.$  
Thus, we conclude that $\frac{s}{k}\in \bigcap_{l\ge 0}S^{ln}\cl M.$ 
 Further, it can be seen that 
\begin{eqnarray*}
\bigcap_{l\ge 0}S^{ln}{\cl M} &\subseteq &
\{f \in \cl M : \exists \ k_f \in  H^{\infty}(z^n) \text{ with }
 k_ff \in \bigcap_{l\ge 0} S^{ln}(\cl M \cap L^2)\}\\
&=& \{ f \in \cl M: \frac{\sum_{j=0}^{n-1} \alpha^{n-(i-1)j} k_ff(\alpha^j z)}{n z^{i-1}}\overline{\phi_{i j}}=0 \ \forall \ i=1,\cdots n\}\\
&=& \{f \in \cl M:\frac{\sum_{j=0}^{n-1} \alpha^{n-(i-1)j} f(\alpha^j z)}{n z^{i-1}}\overline{\phi_{i j}}=0 \ \forall \ i=1,\cdots, n\}\\
&=& \{f \in \cl M:  \sum_{j=1}^n f_j\overline{\phi_{i j}}=0 \ \forall \ i=1,\cdots, n\}
= \cl K_{\cl M}^{\overline{\phi}}
\end{eqnarray*}
where $ f= \sum_{j=1}^n z^{j-1} f_j, \ f_j \in L^p(z^n).$
The rest of the conclusions are now immediate using Case $(1).$\\

\noindent {\bf Case (3).} Lastly, we assume that $p >2.$
It can be verified that $\overline{\cl M}^{L^2}$ is a 
simply invariant under $S^n$ which further yields that \\
\[\overline{\cl M}^{L^2}=\sum_{i=1}^r\oplus \phi_i H^2(z^n) \oplus \cl K_{\overline{\cl M}^{L^2}}^{\overline{\phi}}\]
 where $\{\phi_i\}_{i=1}^r, \ r \le n$ is a set of 
orthonormal vectors in $\overline{\cl{M}}^{L^2}\cap L^{\infty}$ such that each $\phi_i$ does not vanish on a set of positive measure. \\

By virtue of Lemma~\ref{greater than 2}, it is enough to show 
that \[\cl M \subseteq \sum_{i=1}^r \oplus \phi_i H^p(z^n) 
\oplus \cl K_{\overline{\cl M}^{L^2}}^{\overline{\phi}}\] 
Let $f \in \cl M,$ then there exists $ h_i \in H^2(z^n)$ 
and $k \in \cl K_{\overline{\cl M}^{L^2}}^{\overline{\phi}}$ such that
 \[f=\sum_{i=1}^r\phi_i h_i+ k\]  

If we let $k= \sum_{j=1}^nz^{j-1}k_j, \ k_j \in L^2(z^n),$ then
\[ \sum_{i=1}^n(|h_i|^2 + |k_i|^2) = \frac{(\sum_{j=0}^{n-1}|f(\alpha^jz)|^2)}{n} 
\le \tfrac{(\sum_{j=0}^{n-1}|f(\alpha^jz)|)^2}{n}.\]
This implies that
\[|h_i| \le \frac{(\sum_{j=0}^{n-1}|f(\alpha^jz)|)}{n^{1/2}}\]
and
\[|k_i| \le \frac{(\sum_{j=0}^{n-1}|f(\alpha^jz)|)}{n^{1/2}}\]
Hence $h_i \in L^p \cap  H^2(z^n)=H^p(z^n)$ and $ k_i \in L^p(z^n)\ \forall \ i.$ 
Thus, by the Lemma~\ref{greater than 2} we get that $k \in  \cl K_{\cl M}^{\overline{\phi}}$ and   
\[\cl M = \sum_{i=1}^r\oplus\phi_i H^p(z^n)\oplus \cl K_{\cl M}^{\overline{\phi}}.\]
Finally, we conclude the proof of the result by noting that  
\[\cl K_{\cl M}^{\overline{\phi}} = \cl K_{\overline{\cl M}^{L^2}}^{\overline{\phi}} \cap L^p.\]
which follows from the Lemma~\ref{greater than 2}.


\section{Proof of Theorem B.}

 In this section, we shall describe the simply and doubly invariant subspace of $S^n,$ 
 which are contractively contained in $L^p$ for $0< p\le \infty.$ But before we prove 
 the theorem we give some preliminary results which might be of interest in themselves.\\
 \indent To prove the following result which describes the commutant of an operator $S^n,$
 we need to recall the operator $A_{\phi}$ defined in Section 2.

\begin{lemma}\label{commutant}
 The commutant $\{S^n\}^{'}$ of the operator $S^n$ is the set of operators
\[\{A_{\phi}: \phi=(\phi_{ij})\in M_n( L^{\infty}(z^n)), \ 1 \le i,\ j \le n\}\]
\end{lemma}
\begin{proof}
The linear map $U:L^2\rightarrow L^2(z^n)$ defined by $U(z)=z^n$ is an onto isometry and consequently, the map \[V=U\oplus zU\oplus z^{n-1}U:\underbrace{L^2\oplus\cdots\oplus L^2}_{n \text{ times }} \to L^2\] defines an onto isometry. Note that $V^*S^nV=\underbrace{S\oplus\cdots\oplus S}_{n \text{ times }}. $ Thus,
\begin{eqnarray*}\{S^n\}^{'} &=& \left\{VTV^*:T\in \{S\oplus\cdots\oplus S\}^{'}\right\}\\
&=&\left\{V(M_{\psi_{ij}})V^*: (\psi_{ij})\in M_n(L^{\infty})\right\}\\
&=& \left\{ (z^{i-j} M_{\phi_{ij}}): (\phi_{ij}) \in M_n(L^{\infty}(z^n))\right\}
\end{eqnarray*}
where $UM_{\psi_{ij}}U^*= M_{\phi_{ij}}$ and $\phi_{ij} \in L^{\infty}(z^n).$ \\
 \indent Finally, note that $A_{\phi}=(z^{i-j} M_{\phi_{ij}})$ with respect to the decomposition 
$L^2 = L^2(z^n)\oplus z L^2(z^n)\oplus \cdots \oplus z^{n-1}L^2(z^n).$
\end{proof}


We shall now obtain the characterization of the Hilbert spaces boundedly contained in 
$L^q, \ 0 < q \le 2,$ on which $S^n$ acts as a unitary.
In view of Lemma 4.1 above, the ideas of the proof of Theorem 2.1 in \cite{PS} and Theorem 2.3 in \cite{Red} 
can be extended to prove an appropriate generalization of both the results for $S^n.$ 
For the sake of completeness we present the details here. Before we state and prove our result, we need 
another theorem stated in [15, page 3] and proved in \cite{Murray}.

\begin{thm}\label{Redett}
 If $\cl H$ is Hilbert space and $T: \cl H \to L^p, \ 0< p < 2,$ is a 
continuous linear operator, then there exists a continuous linear 
operator $U: \cl H \to L^2$ and a $g > 0$ in $ L^{2p/(2-p)}$
such that $T= M_gU$, where $M_g$ is multiplication by $g$.
\end{thm}

\begin{thm}\label{reducing1}Let $\cl H$ be a Hilbert space boundedly 
contained in $L^q, \ 0 < q \le 2.$ Then $S^n$ acts unitarily 
on $\cl H$ if and only if there exists a function $h$ in 
$L^{2q/(2-q)}$ and $\phi=(\phi_{ij})\in M_n(L^{\infty}(z^n))$ 
such that $\cl H=\cl R(M_{h}A_{\phi})$ isometrically; that is, 
$\|h\|_{\cl H}=\|h\|_{\cl R(M_{h}A_{\phi})}$ for all $h$ in $\cl H.$ 
When $\cl H$ is contractively contained in $L^q,$ we have $\|M_{h}A_{\phi}\|\le 1.$

\end{thm}

\begin{proof}If $\cl H=\cl R(M_{h}A_{\phi})$ isometrically, for some 
$h$ in $L^{2q/(2-q)}$ and $\phi=(\phi_{ij})\in M_n(L^{\infty}(z^n)),$ 
then to prove that $S^n$ acts unitarily on $\cl H,$  it is enough to 
note that $z^nh, \ z^{-n}h \in \cl H \ \forall \ h \in \cl H.$\\ 
To prove the converse we assume that $S^n$ acts unitarily on $\cl H$ 
and let $C: \cl H \to L^q$ denote the bounded containment. 
We now divide the proof in two cases.\\ 
\indent For $q=2,$ there exists a unitary operator 
$U: \cl H \to \cl H$ such that $S^nC=CU.$ This means that $S^nCC^*{S^n}^*=CC^*,$ 
because $U$ is a unitary. Hence, by Lemma~\ref{commutant}, 
$CC^*=A_{\psi}$ for some $\psi \in M_n(L^{\infty}(z^n))$ 
and so by Douglas' factorization theorem $\cl H=\cl R(C)=\cl R((A_{\psi})^{1/2})$.  
Note that $\{A_{\psi}: \phi \in L^{\infty}(z^n)\}$ 
is a $C^*$-subalgebra of $M_n(L^{\infty}),$ thus 
there exists $\phi \in L^{\infty}(z^n)$ 
such that $(A_{\psi})^{1/2}=A_{\phi}.$ It can be 
easily verified that $\|h\|_{\cl H}=\|h\|_{\cl R(A_{\phi})}$ and 
$\|A_{\phi}\| \leq 1$ whenever $\cl H$ is contractively contained in $L^2.$ \\
\indent We now proceed with the case $0<q<2.$ 
By Theorem~\ref{Redett}, there exists a $g > 0$ in $ L^{2q/(2-q)}$
and an operator $U: \cl H \to L^2$ such that $C=M_gU.$ 
Thus, $\cl H=\cl R(C)= \cl R(M_gU)$ 
and since $g >0$ we get that $S^n$ acts unitarily on 
$\cl R(U).$ Now it follows from the above case($q=2$) that 
$\cl R(U)= \cl R(A_{\phi})$ isometrically for some 
$\phi \in M_n(L^{\infty}(z^n))$ such that $UU^*= A_{\phi}^2$ 
with $A_{\phi} \geq 0.$ Finally, we get that  
$\cl H=\cl R(M_{g}A_{\phi})$ isometrically since $g > 0$ and 
$\|M_{g}A_{\phi}\| \leq 1$ when $\cl H$ is contractively contained in $L^q, 0<q <2.$\\
Note,
\begin{equation*} 
h=\left\{\begin{array}{rl} 1  & {\rm when } \ q=2 \\
g & {\rm  when } \ 0 < q <2
\end{array} \right.
\end{equation*}
\end{proof}

\begin{remark}
Note that the above theorem gives us an alternate characterization of 
the reducing part in the classical case of $L^2.$
\end{remark}

As an immediate consequence we get the following analogue of 
Corollary 2.2 in \cite{PS} and Theorem 2.4 in \cite{Red}.
\begin{cor}\label{reducing2}Let $\cl H$ be boundedly contained in 
$L^q, \ 0< q \le 2$ and let $S^n$ act unitarily on $\cl H.$ Then 
$\cl H\cap L^{\infty}\ne 0$ if and only if $\cl H\ne 0.$ 
\end{cor}

\noindent The following result completes the characterization of 
Hilbert spaces boundedly contained in $L^p, \ 0 < p \le \infty,$ on which $S^n$ acts unitarily.
\begin{thm}
If $\cl H$ is a Hilbert space boundedly contained in $L^q, \ 2 <q \le \infty,$ on which 
$S^n$ acts unitarily, then $\cl H= \{0\}.$
\end{thm}
\begin{proof}If $\cl H$ is boundedly contained in $L^q$ for $ q>2,$ 
then $\cl H$ is boundedly contained in $L^2,$ since $L^q\subseteq L^2$ 
and the $L^q$-norm dominates the $L^2$-norm. Thus by Theorem \ref{reducing1},  
$\cl H=\cl R(M_{\mu}A_{\phi})$ for some function $\phi=(\phi_{ij})\in M_n{L^{\infty}(z^n)}.$ 
Then for each fixed $j, \ 1\le j\le n,$ the function  $\sum_{i=1}^n z^{i-1}\phi_{ij}$ 
multiplies $L^2(z^n) $ into $\cl H$ , and hence multiplies $L^2$ into $L^q.$ 
However, there is no nonzero function in $L^{\infty}$ that multiplies 
$L^2$ into $L^q.$ Therfore, $\sum_{i=1}^n z^{i-1}\phi_{ij}=0, \ \forall \  1\le j\le n.$ 
This implies, $\phi_{ij}=0 \ \forall \ 1\le i,j\le n$ and hence $\cl H=0.$
\end{proof}

We are now in a position to characterize the Hilbert spaces, simply 
invariant under $S^n$ which are contractively contained in 
$L^q, \ 0<q\le{\infty}.$ We begin with the following lemma.
\begin{lemma}{\label{DB ortho}}
Let $\cl M \neq 0 $ be a 
Hilbert space, simply invariant under $S^n$, 
contractively contained in $L^q, \ 0 < q \le 2$ and on which 
$S^n$ acts isometrically. Let $\cl N=\cl M \ominus S^n (\cl M),$ 
then
\begin{itemize}
 \item [(i)] $\cl N \subseteq L^{2q/(2-q)},$
\item[(ii)]for every $\phi \in \cl N$ we have that 
$\|\phi f\|_{\cl M}= \|f\|_2 \ \forall f \in H^2(z^n),$
\item[(iii)]elements of $\cl N$ cannot vanish on a set of positive measure.
\end{itemize}

\end{lemma}
\begin{proof}
By using the decomposition of isometries\cite[page 109]{Hof1}, we may write 
\[\cl M= \cl M_1\oplus \bigcap_{k \ge 0} S^{kn}\cl M\]
where $\cl M_1= \sum_{k=0}^{\infty}S^{kn}(\cl N).$
In view of the hypothesis that $\cl M$ is simply invariant, $\cl M_1 \neq 0.$
Consequently, we see that $\cl N\neq 0$ and hence choose an arbitrary element 
$\phi \in \cl N$ with $\|\phi\|_{\cl M}=1.$ Then $\{\phi z^{kn}\}_{k\ge0}$ is an orthonormal sequence in $\cl M.$\\
Let $f\in H^2(z^n)$ then $f=\sum_{k=0}^\infty\alpha_k z^{kn}$ and $f_l=\sum_{k=0}^l \alpha_k z^{kn}$ 
converges to f in $L^2.$
Note that 
\[ \|\phi f_l\|_{\cl M}^2 = \sum_{k=0}^l|\alpha_k|^2= \|f_l\|^2, \]
since  $\{\phi z^{kn}\}_{k\geq 0}$ is an orthonormal set in $\cl M. $\\
Thus, it follows that $\{\phi f_l\}_l $ is a Cauchy sequence in M which 
further implies that there exists $g \in \cl M$ such that 
$\phi f_l \longrightarrow g$  in $\cl M$,
and so in $L^2.$\\ Consequently, there exists a subsequence 
$ \{\phi f_{l_j}\}$ such that 
the sequence $\phi f_{l_j} \rightarrow g$  
and $ f_{l_j} \rightarrow  f$  almost everywhere.  
This shows that $\phi f = g $  a.e. and hence $\phi f \in \cl M.$ 
Thus, we can conclude that $\phi H^2(z^n) \subseteq \cl M$ and 
hence $\phi \in L^{2q/(2-q)}.$ Also, $\|\phi f\|_{\cl M}= \|f\|_2 \  \forall \ f \in H^2(z^n).$\\

\indent We now show that no element of $N$ can vanish on a set of positive measure unless 
it is zero. Let $\phi\in \cl N$ and suppose it is zero on a set $A$ of positive 
measure. Then as in Lemma ~\ref{ortho} we get an unbounded sequence 
$\{h_l\}$ in $H^{\infty}(z^n)$ such that the sequence $\{h_l\phi\}$ is in 
$\cl M\cap L^{\infty}$ and \[\|h_l\|_2=\|h_l\phi\|_{\cl M}\le \delta \|h_l\phi\|_p\]
By construction of $h_l,$ the right hand side is bounded by a constant 
independent of $l$ and the left hand side is unbounded. Thus we get a 
contradiction if $\phi$ is non zero. This completes the proof.
\end{proof}

\noindent{\bf{\em Proof of Theorem B.}} \ It follows from the above 
lemma and the hypotheses that $\cl M\cap L^p$ is a non-zero closed 
subspace of $L^p$ and is simply invariant under $ S^n$. Thus by Case (3) 
of Theorem A, there exists a set of orthonormal vectors 
$\{\phi_i\}_{i=1}^r$ in $(\overline{\cl M \cap L^p}^{L^2}) \ominus S^n( \overline{\cl M\cap L^p}^{L^2})$ 
such that $ r \leq n$ and \[\cl M\cap L^p = \sum_{i=1}^r\oplus \phi_i H^p(z^n)\oplus\cl K_{\cl M \bigcap L^p}^{\overline{\phi}}\] 
where $\cl K_{\cl M \cap L^p}^{\overline{\phi}}=\bigcap_{k \ge 0}S^{kn}(M\cap L^p).$ \\

\indent Assume for the moment that $\{\phi_i\}_{i=1}^r$ is a basis 
of $({\cl M\cap L^p}/{S^n(\cl M\cap L^p)})$ so that dim$({\cl M\cap L^p}/{S^n(\cl M\cap L^p)})=r.$ 
Then we claim that the dimension of $\cl N=\cl M\ominus S^n\cl M$ is at most $r$. \\ 
Let $\{ \psi_1,\psi_2,\cdots,\psi_r, \psi_{r+1}\}$ be
an orthonormal set in $\cl N \subseteq \cl M.$
It follows from the Lemma ~\ref{DB ortho}(i) that $\psi_i + S^n(\cl M\cap L^p)  \in {\cl M\cap L^p}/{S^n(\cl M\cap L^p)} \ \forall \ i.$\\
Suppose $\sum_{i=1}^{r+1}\alpha_i(\psi_i+ S^n(\cl M\cap L^p)) = S^n(\cl M\cap L^p).$ \\
Then $\alpha_1\psi_1+\cdots+\alpha_r\psi_r+\alpha_{r+1}\psi_{r+1} \in S^n(\cl M\cap L^p) \subseteq  S^n(\cl M)$  
which implies, $\alpha_1\psi_1+\cdots+\alpha_r\psi_r+\alpha_{r+1}\psi_{r+1} = 0.$ 
Thus $\alpha_i = 0, \  i=1,\cdots,r+1.$ 
This shows that $\{\psi_i + S^n(\cl M\cap  L^p)\}_{i=1}^{r+1} $ is a linearly
independent set in ${\cl M\cap L^p}/{ S^n(\cl M\cap L^p)},$ 
which is a contradiction to the assumption. Thus, the dimension of $\cl N \le r.$\\
 
\indent Finally, to show that $\{\phi_i\}_{i=1}^r$ is a basis of $({\cl M\cap L^p}/{S^n(\cl M \cap L^p)})$
we let $\sum_{i=1}^r \alpha_i(\phi_i+S^n(\cl M\cap  L^p)) = S^n(\cl M\cap  L^p).$  
This implies that $\sum_{i=1}^r \alpha_i \phi_i \in  S^n(\cl M\cap  L^p) 
\subseteq {S^n(\overline{\cl M\cap  L^p}^{L^2})}$ which further yields 
$\sum_{i=1}^r\alpha_i \phi_i=0.$ Since $\{\phi_i\}_{i=1}^r$ 
is a linearly independent set, therefore each $\alpha_i = 0.$ 
Thus, $\{\phi_i + S^n(\cl M\cap L^p)\}_{i=1}^r$ is a linearly
independent set in  $({\cl M\cap  L^p}/{S^n(\cl M\cap L^p)}).$ \\

Let $f \in \cl M\cap L^p,$ with $ f= \sum_{i=1}^r\phi_ih_i+g, $  
for some $ h_i\in H^2(z^n), \  g\in \cl K_{\cl M \cap L^p}^{\overline{\phi}}=\bigcap_{k \ge 0}S^{kn}(M\cap L^p).$   Then \[f+S^n(\cl M\cap L^p)=\sum_{i=1}^r\phi_ih_i+S^n(\cl M\cap L^p)\]
If we let $h_i = \sum_{k=0}^{\infty}\alpha_k^i z^{kn}$ then
$\phi_i(\sum_{k=0}^l\alpha_k^iz^{kn}) \rightarrow  \phi_ih_i $ in $L^p \ \forall \ i$, as $l\rightarrow\infty$, since $\phi_i \in  L^\infty \ \forall \ i.$
Note that \[ \phi_i h_i + \cl S^n(\cl M\cap  L^p) = \lim_l\alpha_0^i(\phi_i + S^n(\cl M\cap  L^p)) = \alpha_0^i(\phi_i +  S^n(\cl M\cap  L^p))\]
Thus, $f+{S}^n(\cl M\cap L^p) = \sum_{i=1}^r\alpha_0^i\phi_i+ S^n(\cl M\cap L^p).$ 
Hence $\{\phi_i+ S^n(\cl M\cap L^p)\}_{i=1}^r$ is a basis of 
$({\cl M\cap L^p}/{S^n(\cl M\cap L^p)}).$\\

Let $\{\psi_i \}_{i=1}^s$ with $ \ s\le r$ be
an orthonormal basis of $\cl N,$ then by using 
the decomposition of isometries \cite[page 109]{Hof1} we get that 
$$\cl M =\sum_{i=1}^s\oplus \psi_i H^2(z^n) \oplus \bigcap_{k\geq 0} S^{kn}(\cl M).$$
It follows from the Theorem~\ref{reducing1} that 
$\bigcap_{k\geq 0} S^{kn}(\cl M)= \cl R(M_{g}A_{\phi})$ 
where $\mu$ and $\phi$ have the desired properties. Hence, 
we obtain the required decomposition of $\cl M.$\\

Note that (ii) follows from the above lemma. 
Thus, it only remains to establish (iii). For the 
sake of notational simplicity, we prove (iii) for $n=2,$ 
the proof for general $n$ is identical. Fix an integer $i, \ 1\le i\le n$ and let $\psi_i=\psi_{i1}+z\psi_{i2}$ where $\psi_{ij} \in L^{\infty}(z^2).$ We shall divide the proof for (iii) in the following cases. \\

\noindent {\bf Case(1).} Suppose $p=2.$ Then the inequality 
\[\|h\|_2=\|\phi_ih\|_{\cl M}\le \delta\|\phi_ih\|_2\] 
implies that
\[\tfrac{1}{2\pi}\int_{\bb T}({\delta}^2(|\psi_{i1}|+|\psi_{i2}|)^2-1)|h|^2dm\ge 0\]
for all trigonometric polynomials $h,$ from which it follows that \[\|(|\psi_{i1}|+|\psi_{i2}|)^{-1}\|_{\infty}\le \delta\] and hence $|\psi_{i1}|+|\psi_{i2}|)^{-1}\in L^{\infty}.$\\

\noindent {\bf Case(2).} We now assume that $2< p < \frac{2q}{2-q}.$
For each positive integer $m,$ we define 
\[ E_m=\{z: |\psi_{i1}(z)|>\tfrac{1}{m} \ {\rm or}  \ |\psi_{i2}(z)|>\tfrac{1}{m}\}\]
and for a fixed real $r,$
\begin{equation*} 
k_m=\left\{\begin{array}{rl} r{\rm log}(|\psi_{i1}(z)|+|\psi_{i2}(z)|)  &  {\rm on} \ E_m\\
 0 & {\rm  on } \ E_m^c
\end{array} \right.
\end{equation*}
Then \[h_m={\rm exp}(k_m+i\tilde{k}_m)\] belongs to 
$H^{\infty}(z^2)$ where $\tilde{k}_m$ denote the harmonic 
conjugate of $k_m.$ \\

\noindent {\bf Subcase 2(a).} Suppose $2<p<\frac{2q}{2-q}.$ 
In this case we obtain the inequality
\begin{eqnarray*}
\left(\tfrac{1}{2\pi}\int_{E_m}(|\psi_{i1}|+|\psi_{i2}|)^{2r}dm\right)^{1/2}&\le & \|h_m\|_2\\
&=& \|\psi_ih_m\|_{\cl M}\\
&\le & \delta\left(\int_{\bb T}|\psi_i|^p|h_m|^p dm\right)^{1/p}\\
&\le &\delta\left(\int_{\bb T}|(\psi_{i1}|+|\psi_{i2}|)^p|h_m|^p dm\right)^{1/p}.
\end{eqnarray*}
Letting $m\rightarrow \infty$ in the above set of inequalities we obtain
\[\left(\tfrac{1}{2\pi}\int_{E_m}(|\psi_{i1}|+|\psi_{i2}|)^{2r}dm\right)^{1/2}\le \delta\left(\int_{\bb T}|(\psi_{i1}|+|\psi_{i2}|)^p|h_m|^p dm\right)^{1/p}.\]
Choosing $r=\tfrac{p}{2-p}<0$ we have $2r=p(1+r)=-s.$ Hence the 
above inequality yields
\[\left(\tfrac{1}{2\pi}\int_{\bb T}(|\psi_{i1}|+|\psi_{i2}|)^s\right)^{\tfrac{1}{2}-\tfrac{1}{p}}\le \delta,\]
that is, $(|\psi_{i1}|+|\psi_{i2}|)^{-1}\in L^{s},$ where 
$s=\tfrac{2p}{p-2}.$\\

\noindent {\bf Subcase 2(b).} Here we take $\ p=\frac{2q}{2-q}, \ q\ne 2.$ 
Then we can apply the same arguments as in the case $2<p<\frac{2q}{2-q},$ 
to obtain $(|\psi_{i1}|+|\psi_{i2}|)^{-1}\in L^{s},$ where $s=\tfrac{2p}{p-2}.\\$

\noindent {\bf Subcase 2(c).} Finally, we consider the case when 
$\ p=\frac{2q}{2-q}, \ q = 2.$ We find, upon taking the limits in 
the inequality \[\left(\tfrac{1}{2\pi}\int_{E_m}(|\psi_{i1}|+|\psi_{i2}|)^{2r}dm\right)^{1/2}\le \delta\|(|\psi_{i1}|+|\psi_{i2}|)h_m \|_{\infty}\]
 and by using the easily deduced fact that 
$\|(|\psi_{i1}|+|\psi_{i2}|)h_m\|_{\infty}\rightarrow 1$ as $m\rightarrow \infty,$ that 
\[|(\psi_{i1}|+|\psi_{i2}|)^{-1}\in L^2.\]

\noindent Lastly, to prove the final assertion in the theorem, note 
that $\bigcap_{k\ge0}S^{kn}(\cl M)$ is contractively contained in 
$L^2$ and $S^n$ acts as an unitary on it. If we now suppose 
$\bigcap_{l\ge0}S^{nl}(\cl M)\ne 0,$ then by corollary ~\ref{reducing2}, 
$\bigcap_{k\ge0}S^{kn}(\cl M \bigcap L^{\infty})\ne 0,$ which further 
implies that $\bigcap_{k\ge0}S^{kn}(\cl M\cap L^p)\ne 0.$ But this 
cannot happen by case (3) of Theorem A unless $s < n.$ 
This completes the proof of the theorem.\\

The following result comes as a consequence of the above theorem which asserts that, 
under the same conditions as in the above theorem, there are no non-trivial simply 
invariant subspaces contractively contained in $L^r$ for some $r > 2.$ 
We prove this result by using the same arguments as used by Paulsen and Singh 
to prove their corollary 5.2 in \cite{PS}. 

\begin{cor}
Let $\cl M$ be a simply invariant Hilbert space boundedly
contained in $L^r$ for some $r > 2.$ Suppose that $S^n$ acts isometrically 
on $\cl M$ and that $\cl M$ satisfies the condition that there exists $2 \le p \le \infty$
and $\delta >0 $ such that
$\|f\|_{\cl M} \le \delta \|f\|_p \  \forall \ f \in \cl M\cap L^p.$
Then $\cl M=\{0\}.$
\end{cor}
\begin{proof}
Assume that $\cl M \neq 0.$ Then $\cl M$ is a non zero simply invariant 
Hilbert space contractively contained in $L^2$ since $L^r \subset L^2.$
Also, it satisfies the condition as in the above theorem by the hypotheses 
and thus we conclude that there exists $\phi \in L^{\infty}$ such that $\phi H^2 \subseteq \cl M.$
However, for each $\phi \in L^{\infty}$ there exists $f \in H^2$ such that 
$ \phi f \not\in L^r,$ thus a contradiction. This completes the proof of the result.
\end{proof}
{\em Acknowlegements: The authors thank Prof. Vern Paulsen for valuable discussions. The third author expresses appreciation to the 
Department of Mathematics, University of Houston, for support where a part of this paper was written. He also acknowledges the 
facilities provided by the Mathematical Sciences Foundation.}

\end{document}